\documentclass[a4paper,10pt]{article}

\usepackage{amsmath,amssymb,amsthm}
\usepackage{enumerate}
\usepackage{latexsym}
\usepackage{mathabx}

\theoremstyle{definition}
  \newtheorem{defn}{Definition}[section]

\theoremstyle{plain}
  \newtheorem{theorem}[defn]{Theorem}
  \newtheorem{proposition}[defn]{Proposition}
  \newtheorem{corollary}[defn]{Corollary}
  \newtheorem{lemma}[defn]{Lemma}

\DeclareMathOperator{\Aut}{Aut}
\DeclareMathOperator{\Con}{Con}
\DeclareMathOperator{\End}{End}
\DeclareMathOperator{\Hom}{Hom}
\DeclareMathOperator{\id}{id}
\DeclareMathOperator{\Mat}{Mat}
\DeclareMathOperator{\Res}{Res}

\newcommand{\bA}{\mathbf{A}}
\newcommand{\bB}{\mathbf{B}}
\newcommand{\bL}{\mathbf{L}}
\newcommand{\bK}{\mathbf{K}}
\newcommand{\bM}{\mathbf{M}}
\newcommand{\bN}{\mathbf{N}}
\newcommand{\bP}{\mathbf{P}}
\newcommand{\bR}{\mathbf{R}}

\makeatletter
\renewcommand*{\@fnsymbol}[1]{}
\makeatother

\begin{document}

\title{Invertible Matrices over Finite Additively Idempotent Semirings%
  \thanks{This work has been supported by Science Foundation Ireland under
    grant no.\ 08/IN.1/I1950.}}

\date{}

\author{%
  Andreas Kendziorra\\
 \begin{small}Claude Shannon Institute\end{small}\\
  \begin{small}University College Dublin\end{small}
  \and
  Stefan E.\ Schmidt\\
  \begin{small}Institut f\"ur Algebra\end{small}\\
  \begin{small}Technische Universit\"at Dresden\end{small}
  \and
  Jens Zumbr\"agel\\	
  \begin{small}Claude Shannon Institute\end{small}\\
  \begin{small}University College Dublin\end{small}}

\maketitle

\begin{abstract}
  We investigate invertible matrices over finite additively idempotent
  semirings.  The main result provides a criterion for the
  invertibility of such matrices.  We also give a construction of the
  inverse matrix and a formula for the number of invertible
  matrices.\medskip

  \noindent\textit{Keywords:} Matrix inversion, Semirings, Lattices.\medskip

  \noindent\textit{2010 Mathematics Subject Classification:} 15A09,
  15B99, 16Y60, 06B10.
\end{abstract}

\section{Introduction}

Monico, Maze, and Rosenthal generalized in~\cite{MMR} the
Diffie-Hellman protocol, which is used in public-key cryptography, by
using arbitrary semigroup actions instead of the group exponentiation.
Some of the proposed actions involve matrices over proper finite
simple semirings with zero.  Monico showed in~\cite{monico} that these
semirings are additively idempotent and Zumbr\"agel presented
in~\cite{zum} a characterization of these semirings, which can be
formulated by residuated mappings of finite lattices.  When matrices
are used for cryptographic purposes, the principal questions arise how
to easily decide whether a matrix is invertible and, if so, how to
compute the inverse matrix. For matrices over fields the answers are
well-known: A matrix over a field is invertible iff its determinant is
nonzero, and the inverse of an invertible matrix can be computed, e.g.,
with the help of Gauss-Jordan elimination.  A similar useful criterion
for invertible matrices over arbitrary semirings is not known in
general.  There are results for invertible matrices over boolean
algebras \cite{luce, rutherford, wedderburn}. Furthermore, there exist
generalizations to matrices over certain ordered algebraic structures
\cite{blyth}, and there are results for matrices over Brouwerian
lattices \cite{zhao} and distributive lattices \cite{giveon}. Also for
matrices over certain commutative semirings some results are known
\cite{dolzanoblak, tan}.

In this paper we present a criterion for invertible matrices over
finite additively idempotent semirings with zero and one.  As an
important consequence, for a finite additively idempotent base
semiring with irreducible additive semigroup we get that a matrix is
invertible iff it is a generalized permutation matrix.  Besides the
criterion, we present a construction for the inverse of an invertible
matrix and a formula for the number of invertible matrices of a given
size over a given semiring.  For these results we represent a finite
additively idempotent semiring with zero and one as a semiring of
residuated mappings of a finite lattice.  The invertibility criterion
is then based on a description of automorphisms of lattices.  The
results cover the case of invertible matrices over proper finite
simple semirings with zero, which are used in~\cite{MMR}.

\section{Matrices over additively idempotent semirings\!}

\begin{defn}
  Let $R$ be a nonempty set and $+$ and $\cdot$ two binary operations
  on $R$.  Then $(R,+,\cdot)$ is called a \textit{semiring} if $(R,+)$
  is a commutative semigroup, $(R,\cdot)$ is a semigroup and the
  distributive laws $r\cdot(s+t)=r\cdot s+r\cdot t$ and $(r+s)\cdot
  t=r\cdot t+ s\cdot t$ for all $r,s,t\in R$ hold.  If a neutral
  element $0$ of the semigroup $(R,+)$ exists and it satisfies $0\cdot
  x=x\cdot 0=0$ for all $x\in R$, then it is called a \textit{zero}.
  If a neutral element $1$ of the semigroup $(R,\cdot)$ exists, then
  it is called a \textit{one}.  A semiring is called a \textit{proper}
  semiring if it is not a ring, i.e., $(R,+)$ is not a group.
\end{defn}

For invertible matrices over semirings we clearly have just to
consider semirings with zero and one.

A \textit{lattice} $\bL=(L,\leq)$ is an ordered set where for every
two elements $x,y\in L$ the supremum $x\vee y$ and the infimum
$x\wedge y$ in $L$ exists. $\bL$ is called \textit{complete} if for
every subset $X\subseteq L$ the supremum $\bigvee X$ and the infimum
$\bigwedge X$ in $L$ exists. A complete lattice has a greatest element
$1_\bL$ and a least element $0_\bL$.

There exists an equivalent definition of lattices as algebras: A
\textit{lattice} is an algebra $(L,\vee,\wedge)$, where $L$ is a
nonempty set, and $\vee$ and $\wedge$ are binary, associative,
commutative operations on $L$, which fulfill the absorption laws
${x\vee(x\wedge y)=x}$ and ${x\wedge(x\vee y)=x}$ for every $x,y\in
L$.  That these two definitions are equivalent can be found in
\cite{graetzer}.

If $\bL$ and $\bK$ are complete lattices and a mapping $f:L\rightarrow
K$ fulfills $f(\bigvee X)=\bigvee f(X)$ for every subset $X\subseteq
L$, then $f$ is called \textit{residuated} (residuated mappings are
usually defined more generally for arbitrary ordered sets, but for
complete lattices this definition is sufficient; see \cite{janowitz}).
If $\bL$ and $\bK$ are finite then $f:L\rightarrow K$ is residuated
iff $f(x\vee y)=f(x)\vee f(y)$ for every $x,y\in L$ and
$f(0_\bL)=0_\bK$.  By $\Res(\bL)$ we denote the set of all residuated
mappings from $\bL$ to $\bL$. The structure $(\Res(\bL),\vee,\circ)$,
where $\vee$ denotes the pointwise supremum and $\circ$ the
composition of two mappings, is a semiring. Further the mapping
$\mathbf{0}:L\rightarrow L,\ x\mapsto 0_\bL$, is a zero and $\id_L$ a
one of this semiring. More information about lattices can be found in
\cite{birkhoff,graetzer} and about residuated mappings in
\cite{janowitz}.

If $(R,+)$ is a commutative idempotent semigroup, then $(R,\leq)$ with
${x\leq y}\,:\Leftrightarrow\,{x+y=y}$ is a semilattice with the
supremum operation $\vee=+$. If $(R,+)$ is further finite and has a
neutral element then $(R,\leq)$ is even a lattice (see
\cite{birkhoff}). Hence, if $(R,+,\cdot)$ is a finite additively
idempotent semiring with zero, then $(R,\leq)$ is a lattice. The next
proposition shows that one can embed such a semiring into a semiring
of residuated mappings if it has additionally a one.

\begin{proposition}\label{prop_add_idemp_semirings}
  Let $(R,+,\cdot)$ be a finite additively idempotent semiring with
  zero and one, $\bR:=(R,\leq)$ and
  \begin{displaymath}
    T: R\rightarrow \Res(\bR) \:, \quad r\mapsto T_r
    \quad \text{with}\quad T_r: x\mapsto rx \:.
  \end{displaymath}
  Then $(R,+,\cdot)$ is isomorphic to the subsemiring
  $(T(R),\vee,\circ)$ of $(\Res(\bR),\vee,\circ)$.
\end{proposition}

\begin{proof}
  Clearly, $T_r\in \Res(\bR)$ for every $r\in R$ and $T$ is a semiring
  homomorphism between $(R,+,\cdot)$ and $(\Res(\bR),\vee,\circ)$.
  Since $(R,+,\cdot)$ has a one $1$, we have that $T_r=T_s$ implies
  $r=T_r(1)=T_s(1)=s$ for all $r,s\in R$, i.e., $T$ is injective.
  Hence, $(R,+,\cdot)$ is isomorphic to the subsemiring
  $(T(R),\vee,\circ)$ of $(\Res(\bR),\vee,\circ)$.
\end{proof}

Next we present the characterization of proper finite simple semirings
with zero by Zumbr\"agel \cite{zum}, which was in combination with
\cite{MMR} the motivation for this work.

\begin{defn}
  Let $\bA=(A,F)$ be an algebra. A \textit{congruence} on $\bA$ is an
  equivalence relation $\theta$ on $A$ with the following property:
  For every $n\in \mathbb{N}$ and every $n$-ary mapping $f\in F$ and
  elements $a_i,b_i\in A$ with $a_i\theta b_i$ for $1\leq i\leq n$, it
  holds that $f(a_1,...,a_n)\theta f(b_1,...,b_n)$.
\end{defn}

For a homomorphism $f:\bA\rightarrow \bB$ of some algebras $\bA, \bB$
of the same type the \textit{kernel} $\ker(f):=\{(a,a')\in A\times
A\mid f(a)=f(a')\}$ of $f$ is a congruence on~$\bA$. Also the
\textit{equality relation} $\Delta_A:=\{(a,a) \mid a\in A\}$ and the
\textit{complete relation} $\nabla_A:=A\times A$ are congruences on
$\bA$.  Congruences are one of our main tools to derive our
results. For a wider background on congruences and universal algebra
see \cite{graetzer_UA}.

\begin{defn}
  A semiring $(R,+,\cdot)$ is called \textit{simple} if its only
  congruences are $\Delta_R$ and $\nabla_R$.
\end{defn}

For a complete lattice $\bL=(L,\leq)$ and $a,b\in L$ define a mapping
$e_{a,b}\in\Res(\bL)$ by
\begin{displaymath}
  e_{a,b}:L\rightarrow L \:, \quad x\mapsto \begin{cases}
    0_\bL\quad &\text{if }x\leq a,\\
    b &\text{otherwise}.
  \end{cases}
\end{displaymath}
The main result from \cite{zum} can be stated as follows:

\begin{theorem}\label{theorem_zum}
  Let $\bL$ be a finite lattice and $(R,\vee, \circ)$ a subsemiring of
  $(\Res(\bL),\vee, \circ)$ such that $e_{a,b}\in R$ for every $a,b\in
  L$. Then $(R,\vee,\circ)$ is a proper finite simple semiring with
  zero.  Conversely, every proper finite simple semiring $(S,+,\cdot)$
  with $|S|>2$ and a zero is isomorphic to such a semiring.
\end{theorem} 

For two monoids $\bM$ and $\bN$ we denote by $\Hom(\bM,\bN)$ the set
of all monoid homomorphism from $\bM$ to $\bN$ and by $\End(\bM)$ the
set of all endomorphisms of~$\bM$.  Let $I$ be a finite index set and
$\bM_i$ commutative monoids for every $i\in I$.  It is easy to see
that the mapping
\begin{align*}
  \Omega: \bigtimes_{(i,j)\in I\times I}\!\! \Hom (\bM_j,\bM_i) 
  \rightarrow \End \Big( \bigtimes_{i\in I} \bM_i \Big) \:, \quad
  (f_{i,j})\mapsto \Big(\sum_{j\in I} f_{i,j} \Big)_{i\in I}
\end{align*}
with 
\begin{displaymath}
 \Big(\sum_{j\in I} f_{i,j} \Big)_{i\in I} \big((m_j)_{j\in I}\big) = 
 \Big(\sum_{j\in I} f_{i,j} (m_j)\Big)_{i\in I}
\end{displaymath}
for every $(m_j)_{j\in I}\in \bigtimes_{j\in I}\bM_j$ is an
isomorphism between the semirings
\begin{displaymath}
  \Big( \bigtimes_{(i,j)\in I\times I}\!\! \Hom (\bM_j,\bM_i), +, \cdot \Big)
  \quad\text{ and }\quad 
  \Big( \End \big( \bigtimes_{i\in I} \bM_i \big), +, \circ \Big) \:,
\end{displaymath}
where $+$ denotes in each case the pointwise sum, $\circ$ on
$\End(\bigtimes_{i\in I}\bM_i)$ the composition, and $\cdot$ is
defined on $\bigtimes_{(i,j)\in I\times I} \Hom (\bM_j,\bM_i)$ by
$(f_{i,j})\cdot (g_{i,j}) =: (h_{i,j})$ with $h_{i,j}=\sum_{k\in
  I}f_{i,k}\circ g_{k,j}$.  In particular, it holds that
\begin{align*}
  \big( \Mat_{I\times I}(\End(\bM)),+,\cdot \big) :=
  \Big( \bigtimes_{(i,j)\in I\times I}\!\! \End(\bM), +, \cdot \Big) \cong 
  \big( \End(\bM^I), +, \circ \big) \:,
\end{align*}
where $\bM^I:=\bigtimes_{i\in I}\bM$.

If $\bL$ and $\bK$ are finite lattices and $f:L\rightarrow K$ is a
mapping, then $f$ is residuated iff $f$ is a monoid homomorphism
between the monoids $(L,\vee,0_\bL)$ and $(K,\vee,0_\bK)$.  Hence, we
get
\begin{displaymath}
  \big( \Mat_{I\times I}(\Res(\bL)), +, \cdot \big) \cong
  \big( \Res(\bL^I), \vee, \circ \big) \:.
\end{displaymath}
Therefore, a matrix $M=(m_{i,j})\in \Mat_{I\times I}(\Res(\bL))$ is
invertible iff the corresponding residuated mapping
\begin{displaymath}
 \varphi_M:=\big(\bigvee_{j\in I}m_{i,j}\big)_{i \in I}\in
\Res(\bL^I)
\end{displaymath}
is invertible, which is equivalent to $\varphi_M$ being bijective.

A mapping $f:P\rightarrow Q$ between ordered sets $(P,\leq)$ and
$(Q,\leq)$ is an \textit{(order) isomorphism} if $f$ is surjective and
it fulfills $x\leq y\Leftrightarrow f(x)\leq f(y)$ for all $x,y\in
P$. Note that an order isomorphism is automatically injective. An
order isomorphism from $(P,\leq)$ to ($P,\leq)$ is called
\textit{(order) automorphism}.  Note in the following that the
concepts of isomorphisms and automorphisms of lattices as ordered sets
and as algebras are equivalent (see \cite{graetzer}).

\begin{lemma}
  Let $\bL$ be a complete lattice and $f\in \Res(\bL)$. Then $f$ is an
  automorphism of $\bL$ iff $f$ is bijective.
\end{lemma}

\begin{proof}
  Let $f$ be bijective.  For $x,y\in L$, the equivalence $x\leq
  y\Leftrightarrow y=x\vee y\Leftrightarrow f(y)=f(x\vee y)=f(x)\vee
  f(y)\Leftrightarrow f(x)\leq f(y)$ holds, i.e., $f$ is an
  automorphism.  The other direction is clear.
\end{proof}

\begin{corollary}
  Let $\bL$ be a finite lattice, $I$ a finite index set, and
  $M=(m_{i,j})\in \Mat_{I\times I}(\Res(\bL))$. Then $M$ is invertible
  iff the corresponding mapping $\varphi_M\in \Res(\bL^I)$ is an
  automorphism of $\bL^I$.
\end{corollary}

Hence, we aim to give a characterization for when a mapping of the
direct product~$\bL^I$ is an automorphism of~$\bL^I$.  If~$\bL$ is a
direct product $\bigtimes_{t\in T}\bL_t$ of irreducible lattices
$\bL_t$, $t\in T$, for a finite index set~$T$, our task is then to
determine when a mapping of the direct product $(\bigtimes_{t\in
  T}\bL_t)^I$ is an automorphism.  Consequently, it suffices to find a
criterion for mappings of direct products of irreducible lattices.  We
present such a criterion (Theorem~\ref{theorem_inv_mat}) and we
translate it so that we can answer the question when a matrix in
$\Mat_{I\times I}(\Res(\bL))$ is invertible
(Corollary~\ref{corollary_inv_mat}). In Section~\ref{sec_subsemirings}
we explain how our results apply to subsemirings of
$(\Res(\bL),\vee,\circ)$, so that, by
Proposition~\ref{prop_add_idemp_semirings}, they can be applied to
every finite additively idempotent semiring with zero and one.

\section{Direct decompositions}

In this section we investigate maximal direct decompositions of
lattices, on which our criterion for matrix invertibility will
crucially depend.

An algebra $\bA=(A,F)$ is called \emph{trivial} if $|A|=1$, otherwise
it is called \emph{nontrivial}. We call an algebra $\bA$
\emph{\mbox{irreducible}} if it is nontrivial and not isomorphic to a
direct product of two nontrivial algebras. Analogously, an ordered set
$\bP=(P,\leq)$ is called \emph{trivial} if $|P|=1$, otherwise it is
called \emph{nontrivial}.  We also call an ordered set $\bP$
\emph{irreducible} if it is nontrivial and not isomorphic to a direct
product of two nontrivial ordered sets.  Clearly, the direct product
of lattices as ordered sets is the same as the direct product of
lattices as algebras. Consequently, a lattice is irreducible as an
ordered set iff it is irreducible as an algebra.

\begin{defn}
  A \textit{subdirect decomposition} of an algebra $\bA$ is a family
  $(\Theta_t)_{ t\in T}$ of congruences of $\bA$ with
  \begin{displaymath}
    \bigcap_{t\in T}\Theta_t = \Delta_A \:.
  \end{displaymath}
  We call a subdirect decomposition $(\Theta_t)_{ t\in T}$ of $\bA$ a
  \textit{direct decomposition} of $\bA$ if the mapping
  \begin{align*}
    \iota :A\rightarrow \bigtimes_{t\in T} A/\Theta_t,\quad a\mapsto
    \big( [a]\Theta_t \big)_{t\in T}
  \end{align*}
  is surjective.  Moreover, we call a direct decomposition
  $(\Theta_t)_{ t\in T}$ of $\bA$ \textit{maximal} if $\Theta_t\neq
  \nabla_A $ for every $t\in T$ and if for every direct decomposition
  $(\Theta_s)_{ s\in S}$ of $\bA$ with $\Theta_s\neq \nabla_A$ for
  every $s\in S$ the inequality $|S|\leq |T|$ holds.
\end{defn}

The mapping $\iota$ is for every algebra $\bA$ and every subdirect
decomposition $(\Theta_t)_{ t\in T}$ of $\bA$ an injective
homomorphism. Therefore, $\bA$ is isomorphic to $\iota(\bA)$. If
$\iota$ is even surjective, then $\bA$ is isomorphic to the direct
product $\bigtimes_{t\in T} \bA/\Theta_t$.  If $\Theta_t$ is non-total
for a $t\in T$, then the factor $\bA/\Theta_t$ is nontrivial.

Let $\bA_i$, $i\in I$, be some nontrivial algebras of the same type
and let $\bA:=\bigtimes_{i\in I}\bA_i$.  For an element $a\in A$, we
denote by $a_i$ the $i$-th coordinate of $a$.  Define the congruence
$\Phi_i:=\{(a,b)\in A\times A\mid a_i = b_i\}$ for every $i\in
I$. Then $(\Phi_i)_{i\in I}$ is clearly a direct decomposition of
$\bA$ and $\Phi_i$ is non-total for every $i\in I$.  Thus for a
maximal direct decomposition $(\Theta_t)_{ t\in T}$ of $\bA$, the
inequality $|T|\geq |I|$ holds.

The next proposition is stated in \cite{Hashimoto}.

\begin{proposition}\label{unique_decomposition}
  The representation of a connected ordered set as the direct product
  of irreducible ordered sets is unique up to pairwise isomorphism of
  the factors.
\end{proposition}

Since a lattice is a connected ordered set, we get the following.

\begin{corollary}\label{corollary_bijection}
  Let $S$ and $T$ be index sets, $\bL_t$ an irreducible lattice for
  every $t\in T$, $\bL :=\bigtimes_{t\in T}\bL_t$, and
  $(\Theta_s)_{s\in S}$ a maximal direct decomposition of $\bL$. Then
  there exists a bijection $\sigma:S\rightarrow T$ with
  $\bL/\Theta_s\cong \bL_{\sigma(s)}$.
\end{corollary}

For this reason, we may assume that if $\bL$ is the direct product of
the irreducible lattices $\bL_t$, $t\in T$, then a maximal direct
decomposition of~$\bL$ is of the form $(\Theta_t)_{t\in T}$ with
$\bL/\Theta_t\cong \bL_t$ for all $t\in T$.

In \cite[Chapter 1.3, Theorem 13]{graetzer} the following result is
proven.

\begin{theorem}\label{congrunce_product}
  Let $\bL$ and $\bK$ be lattices, let $\Theta_L$ be a congruence on
  $\bL$, and let $\Theta_K$ be a congruence on $\bK$.  Define the
  relation $\Theta_L \times \Theta_K$ on $\bL\times \bK$ by
  \begin{displaymath}
    (a,b)(\Theta_L\times \Theta_K) (c,d) \quad\text{iff}\quad 
    a\Theta_L c\ \text{ and }\ b\Theta_K d \:.
  \end{displaymath}
  Then $\Theta_L \times \Theta_K$ is a congruence on $\bL\times
  \bK$. Conversely, every congruence on $\bL\times \bK$ is of this
  form.
\end{theorem}

Note that `$\Theta_L\times \Theta_K$' is a slight abuse of notation,
since it is not identical to the Cartesian product of the two sets
$\Theta_L$ and $\Theta_K$.

It further holds that
\begin{equation}
  [a]\Theta_L\times [b]\Theta_K = 
  \{(c,d)\in L\times K\mid a\Theta_L c\text{ and }b\Theta_K
  d\} = [(a,b)](\Theta_L\times \Theta_K)
  \label{eq_congruence_product}
\end{equation} 
and so
\begin{equation}
  \bL/\Theta_L \times \bK/\Theta_K = 
  (\bL\times \bK)/(\Theta_L\times \Theta_K) \:.
  \label{eq_congruence_product2}
\end{equation}

The following result is a strengthening of Corollary
\ref{corollary_bijection}.

\begin{lemma}\label{lemma_congruences}
  Let $T$ be a finite index set, $\bL_t$ an irreducible lattice for
  every $t\in T$, $\bL :=\bigtimes_{t\in T}\bL_t$, and
  $(\Theta_t)_{t\in T}$ a maximal direct decomposition of $\bL$.  Then
  there exists a permutation $\sigma$ of $T$ with $\bL_{t}\cong
  \bL_{\sigma(t)}$ and
  \begin{displaymath}
    (x_s)_{s\in T}\Theta_{\sigma(t)}(y_s)_{s\in T} \quad\Leftrightarrow\quad
    x_{t} = y_{t}
  \end{displaymath}
  for all $(x_s)_{s\in T},(y_s)_{s\in T}\in L$ and $t\in T$.
\end{lemma}

\begin{proof}
  By Corollary~\ref{corollary_bijection}, we may assume that
  $\bL/\Theta_t\cong \bL_t$ holds for all $t\in T$.  We fix $t_0\in T$
  and define $\bL':=\bigtimes_{t\in T\setminus\{t_0\}}\bL_t$.  Thus,
  $\bL = \bL_{t_0}\times \bL'$. By Theorem~\ref{congrunce_product},
  there exist for every $t\in T$ congruences $\Theta_t^{t_0}\in
  \Con(\bL_{t_0})$, $\Theta_t'\in \Con(\bL')$ with $\Theta_t =
  \Theta_t^{t_0}\times \Theta_t'$. We will show that
  $(\Theta_t^{t_0})_{t\in T}$ is a direct decomposition of
  $\bL_{t_0}$. Let $(x,x')\in\bigcap_{t\in T}\Theta_t^{t_0}$. We have
  to show that $x=x'$ holds.  Let $\bar{y}\in L'$. Thus,
  $(\bar{y},\bar{y})\in \bigcap_{t\in T}\Theta_t'$ and consequently
  $((x,\bar{y}),(x',\bar{y}))\in \bigcap_{t\in T}\Theta_t =
  \Delta_L$. So, we have $(x,\bar{y})=(x',\bar{y})$ and therefore
  $x=x'$. Hence, $(\Theta_t^{t_0})_{t\in T}$ is a subdirect
  decomposition of $\bL_{t_0}$. Now let $x_t\in \bL_{t_0}$ for every
  $t\in T$.  We will show that there exists an element $z\in
  \bL_{t_0}$ with $[z]\Theta^{t_0}_t=[x_t]\Theta^{t_0}_t$ for every
  $t\in T$. Choose an element $(y_t)_{t\in T\setminus \{t_0\}} \in
  \bL'$. For every $s\in T$ we will regard $(x_{s},(y_t)_{t\in
    T\setminus\{t_0\}})\in L$ as the element in $L$, where the
  $t_0$-th coordinate is $x_s$. Since $(\Theta_t)_{t\in T}$ is a
  direct decomposition of $\bL$, there exists an element
  $(\hat{x}_t)_{t\in T}\in L$ with $[(\hat{x}_t)_{t\in T}]\Theta_s =
  [(x_{s},(y_t)_{t\in T\setminus\{t_0\}})]\Theta_s$ for every $s\in
  T$. By Equation~(\ref{eq_congruence_product}), for every $s\in T$,
  \begin{align*}
    [\hat{x}_{t_0}]\Theta_s^{t_0}\times [(\hat{x}_t)_{t\in
      T\setminus\{t_0\}}]\Theta_s'=[(\hat{x}_t)_{t\in T}]\Theta_s
    &= [(x_{s},(y_t)_{t\in T\setminus\{t_0\}})]\Theta_s\\
    &= [x_s]\Theta_s^{t_0}\times [(y_t)_{t\in T\setminus\{t_0\}}]\Theta_s'
  \end{align*}
  holds and it follows that $[\hat{x}_{t_0}]\Theta_s^{t_0} =
  [x_s]\Theta_s^{t_0}$.  Hence, $\hat{x}_{t_0}$ is the desired
  element~$z$ and it follows that $(\Theta_t^{t_0})_{t\in T}$ is a
  direct decomposition of $\bL_{t_0}$.  Consequently, \mbox{$\bL_{t_0}
    \cong \bigtimes_{t\in T} (\bL_{t_0}/\Theta_t^{t_0})$} and since
  $\bL_{t_0}$ is irreducible, there exists a unique $t_1\in T$ with
  $\bL_{t_0}\cong\bL_{t_0}/\Theta_{t_1}^{t_0}$. Thus,
  $\Theta_{t_1}^{t_0}=\Delta_{L_{t_0}}$. By this and Equation
  (\ref{eq_congruence_product2}), it follows that
  \begin{displaymath}
    \bL_{t_0}\times \bL'/\Theta_{t_1}'\cong
    \bL_{t_0}/\Theta_{t_1}^{t_0}\times \bL'/\Theta_{t_1}' =
    (\bL_{t_0}\times \bL')/(\Theta_{t_1}^{t_0}\times
    \Theta_{t_1}')=\bL/\Theta_{t_1}\cong \bL_{t_1}\:.
  \end{displaymath}
  Since $\bL_{t_1}$ is irreducible and $\bL_{t_0}$ nontrivial, we have
  $\bL_{t_0}\cong \bL_{t_1}$ and ${|\bL'/\Theta_{t_1}'|=1}$.  Hence,
  $\Theta_{t_1}' = \nabla_{L'}$. We derive $ (x_t)_{t\in T} \Theta_{t_1}
  (y_t)_{t\in T} \Leftrightarrow x_{t_0} = y_{t_0}$ for all
  $(x_t)_{t\in T} , (y_t)_{t\in T}\in L$.

  We have shown that there exists a mapping $\sigma:T\rightarrow T$
  with $\bL_t\cong \bL_{\sigma(t)}$ and $(x_s)_{s\in T}
  \Theta_{\sigma(t)} (y_s)_{s\in T} \Leftrightarrow x_{t} = y_{t}$ for
  all $(x_s)_{s\in T} , (y_s)_{s\in T}\in L$.  Indeed, with the
  notation above we have $t_1=\sigma(t_0)$.  It remains to show that
  $\sigma$ is injective.  Let $t_2,t_3\in T$ with
  $\sigma(t_2)=\sigma(t_3)$.  There follows the equivalence $
  x_{t_2}=y_{t_2}\Leftrightarrow (x_t)_{t\in T} \Theta_{\sigma(t_2)}
  (y_t)_{t\in T} \Leftrightarrow (x_t)_{t\in T} \Theta_{\sigma(t_3)}
  (y_t)_{t\in T} \Leftrightarrow x_{t_3} = y_{t_3} $ for all
  $(x_t)_{t\in T}$, $(y_t)_{t\in T}\in L$ and we find that $t_2=t_3$.
\end{proof}

\section{Invertible matrices}

\subsection{A criterion}

The following theorem states a criterion for a mapping of a direct
product of irreducible lattices to be an automorphism.  It is
basically a consequence of Lemma~\ref{lemma_congruences}.  We will see
the corresponding result for matrices in
Corollary~\ref{corollary_inv_mat}.

\begin{theorem}\label{theorem_inv_mat}
  Let $T$ be a finite index set, $\bL_t$ an irreducible lattice for
  every $t\in T$, $\bL:=\bigtimes_{t\in T}\bL_t$, and
  $\varphi:L\rightarrow L$ a mapping. Then $\varphi\in \Aut(\bL)$ iff
  there exists a permutation $\sigma$ of $T$ and isomorphisms
  $\varphi_t:L_t\rightarrow L_{\sigma^{-1}(t)}$ for every $t\in T$
  such that
  \begin{displaymath}
    \varphi = (\varphi_{\sigma(t)}\circ \pi_{\sigma(t)})_{t\in T} \:,
  \end{displaymath}
  where $\pi_t$ is the $t$-th projection, i.e., $\varphi((x_t)_{t\in
    T}) = (\varphi_{\sigma(t)}(x_{\sigma(t)}))_{t\in T}$ for all
  $(x_t)_{t\in T}\in L$.
\end{theorem}

\begin{proof}
  Let $\varphi\in \Aut(\bL)$, $\varphi^t:=\pi_t\circ \varphi$ for
  every $t\in T$, and $\Theta_t:=\ker(\varphi^t)$ for every $t\in
  T$. We will show that $(\Theta_t)_{t\in T}$ is a maximal direct
  decomposition of $\bL$. We have
  \begin{align*}
    (x,y)\in \bigcap_{t\in T}\Theta_t &\ \Leftrightarrow\ 
    \forall t\in T: (x,y)\in \Theta_t\\
    &\ \Leftrightarrow\  \forall t\in T: \varphi^t(x)=\varphi^t(y)
    \ \Leftrightarrow\ \varphi(x)=\varphi(y)
    \ \Leftrightarrow\  x=y
  \end{align*}
  for all $x,y\in L$, i.e., $\bigcap_{t\in
    T}\Theta_t=\Delta_L$. Therefore, $(\Theta_t)_{t\in T}$ is a
  subdirect decomposition of $\bL$.  Let $y^t\in L$ for every $t\in
  T$. We will show that there exists a $z\in L$ with
  $[z]\Theta_t=[y^t]\Theta_t$ for every $t\in T$. Let
  $x_t:=\varphi^t(y^t)$ for every $t\in T$, let $x:=(x_t)_{t\in T}$,
  and let $z:=\varphi^{-1}(x)$. It follows that
  $\varphi^t(z)=x_t=\varphi^t(y^t)$ and therefore that $z\Theta_t y^t$
  for every $t\in T$. Hence, $[z]\Theta_t=[y^t]\Theta_t$ for every
  $t\in T$ and $(\Theta_t)_{t\in T}$ is consequently a direct
  decomposition. Since $\varphi$ is bijective, $\Theta_t\neq \nabla_L$
  holds for every $t\in T$. Because a maximal direct decomposition of
  $\bL$ has exactly $|T|$ elements by
  Corollary~\ref{corollary_bijection}, $(\Theta_t)_{t\in T}$ is a
  maximal direct decomposition.

  By Lemma~\ref{lemma_congruences}, there exists a permutation $\sigma$
  of $T$ with $\bL_t\cong \bL_{\sigma(t)}$ and $x\Theta_t
  y\Leftrightarrow x_{\sigma(t)}=y_{\sigma(t)}$ for every $t\in T$ and
  $x,y\in L$. It follows that $\varphi^t(x)= \varphi^t(y)
  \Leftrightarrow x\Theta_t y\Leftrightarrow
  x_{\sigma(t)}=y_{\sigma(t)}$, i.e., $\varphi^t(x)$ depends only on
  $x_{\sigma(t)}$ for every $t\in T$. With
  $\varphi_{\sigma(t)}:=\varphi^t\circ \epsilon_{\sigma(t)}$, where
  $\epsilon_s:L_s\rightarrow L$ is the $s$-th canonical injection,
  it follows that the given criterion is necessary.
  
  The sufficiency of the criterion is trivial. 
\end{proof}

Let $T,I$ be finite index sets, $\bL_t$ an irreducible finite lattice
for every $t\in T$, and $\bL:=\bigtimes_{t\in T}\bL_t$. Then $\bL^I =
\bigtimes_{(t,i)\in T\times I}\bL_{t,i}$, where $\bL_{t,i} = \bL_t$
for every $(t,i)\in T\times I$.  With this notation we derive in the
following the corresponding result for invertible matrices. For a
matrix $A\in \Mat_{I\times I}(\Res(\bL))$, we will denote the $i$-th
row by $A_i$ and we will regard $A_i$ as mapping from $L^I$ to $L$.

\begin{corollary}\label{corollary_inv_mat}
  Let $T,I$ be finite index sets, $\bL_t$ an irreducible finite
  lattice for every $t\in T$, $\bL:=\bigtimes_{t\in T}\bL_t$, and
  $A=(a_{i,j})\in \Mat_{I\times I}(\Res(\bL))$. Then $A$ is invertible
  iff there exists a permutation $\sigma$ of $T\times I$ and an
  isomorphism $\varphi_{t,i}:L_{t,i}\rightarrow L_{\sigma^{-1}(t,i)}$
  for every $(t,i)\in T\times I$ such that
  \begin{displaymath}
    \pi_t\circ A_i= \varphi_{\sigma(t,i)}\circ \pi_{\sigma(t,i)} \:,
  \end{displaymath}
  where $\pi_t$ is the projection from $\bL$ to $\bL_t$ and
  $\pi_{t,i}$ the projection from $\bL^I$ to $\bL_{t,i}$.
\end{corollary}

If $A$ is invertible, then 
\begin{displaymath}
  \varphi_A = (\varphi_{\sigma(t,i)}\circ
  \pi_{\sigma(t,i)})_{(t,i)\in T\times I}
\end{displaymath}
is the corresponding mapping of $A$ in $\Res(\bL^I)$ and $a_{i,j}$ is
of the form $a_{i,j}=(\hat{\varphi}_{i,j,t})_{t\in T}$ with
\begin{displaymath}
  \hat{\varphi}_{i,j,t} = \begin{cases}
    \varphi_{\sigma(t,i)} \quad &\text{if } \exists s\in T:
    \sigma(t,i)=(s,j),\\
    \bar{0}_{\bL_t} &\text{otherwise},
  \end{cases}
\end{displaymath}
where $\bar{0}_{\bL_t}$ is the mapping that maps constantly to
$0_{\bL_t}$.

In the special case where $\bL$ is irreducible, we need not to
consider the index set $T$, since it has just one element.  Then the
equation in Corollary~\ref{corollary_inv_mat} is of the form
$A_i=\varphi_{\sigma(i)}\circ\pi_{\sigma(i)}$ for every $i\in I$,
i.e., $a_{i,\sigma(i)}$ is the only nonzero entry in the $i$-th row
and $a_{i,\sigma(i)}=\varphi_{\sigma(i)}$ holds.  We call a matrix a
\emph{generalized permutation matrix} (or \emph{monomial matrix}) if
each row and each column has exactly one nonzero entry and this
nonzero entry is invertible.

\begin{corollary}\label{cor_gen_perm_mat}
  Let $\bL$ be a finite irreducible lattice, $I$ a finite index set,
  and \mbox{$A\in \Mat_{I\times I}(\Res(\bL))$}. Then $A$ is
  invertible iff $A$ is a generalized permutation matrix.
\end{corollary}

\subsection{Number of invertible matrices}

As another consequence of Theorem~\ref{theorem_inv_mat} we find the
following.

\begin{corollary}
  Let $T$ be a finite index set, $\bL_t$, $t\in T$, pairwise distinct
  irreducible lattices, $e_t\in \mathbb{N}$ for every $t\in T$, and
  $\bL:=\bigtimes_{t\in T}\bL_t^{e_t}$. Then
  \begin{displaymath}
    |\Aut(\bL)| = \prod_{t\in T}e_t!\cdot |\Aut(\bL_t)|^{e_t} \:.
  \end{displaymath}
  In particular, for a finite index set $I$ we have 
  \begin{displaymath}
    |\Aut(\bL^I)| = \prod_{t\in T}(e_t \cdot |I|)!\cdot
    |\Aut(\bL_t)|^{e_t\cdot |I|} \:,
  \end{displaymath}
  which is exactly the number of invertible matrices in $\Mat_{I\times
    I}(\Res(\bL))$.
\end{corollary}

\subsection{The inverse matrix}

The next proposition provides a construction for the inverse matrix of
an invertible matrix.

\begin{proposition}\label{prop_the_inverse}
  Let $T,I$ be finite index sets, $\bL_t$ an irreducible finite
  lattice for every $t\in T$, $\bL:=\bigtimes_{t\in T}\bL_t$, let
  $A=(a_{i,j})\in \Mat_{I\times I}(\Res(\bL))$ be invertible, and
  $\sigma$ and $\varphi_{t,i}$ for every $(t,i)\in T\times I$ as in
  Corollary~\ref{corollary_inv_mat}. Then for the inverse matrix
  $B=(b_{i,j})$ of $A$, the entry $b_{i,j}$ for $i,j\in I$ is of the
  form $b_{i,j}=(\check{\varphi}_{i,j,t})_{t\in T}$ with
  \begin{displaymath}
    \check{\varphi}_{i,j,t} = \begin{cases}
      \varphi_{t,i}^{-1} \quad &\text{if } \exists s\in T:
      \sigma^{-1}(t,i)=(s,j),\\
      \bar{0}_{\bL_t} &\text{otherwise}.
    \end{cases}
  \end{displaymath}
\end{proposition}

\begin{proof}
  As stated before, $\varphi_A = (\varphi_{\sigma(t,i)}\circ
  \pi_{\sigma(t,i)})_{(t,i)\in T\times I}$ is the corresponding
  mapping to $A$ in $\Res(\bL^I)$. The inverse of $\varphi_A$, i.e.,
  the corresponding mapping to the matrix $B$, is the mapping
  $\varphi_B=\varphi_A^{-1} = (\varphi_{t,i}^{-1}\circ
  \pi_{\sigma^{-1}(t,i)})_{(t,i)\in T\times I}$. It follows that
  $b_{i,j}$ is of the form $b_{i,j}=(\check{\varphi}_{i,j,t})_{t\in
    T}$ with $\check{\varphi}_{i,j,t}$ as given in the proposition.
\end{proof}

\subsection{Invertible matrices over subsemirings of
$\Res(\bL)$}\label{sec_subsemirings}

\begin{lemma}
  Let $\bL$ be a finite lattice, $(R,\vee,\circ)$ a subsemiring of
  $(\Res(\bL),\vee,\circ)$, and $\varphi\in R$ such that $\varphi$ is
  invertible in $(\Res(\bL),\circ)$. Then $\varphi^{-1}\in R$.
\end{lemma}

\begin{proof}
  Since $\varphi$ is invertible and $\bL$ is finite, we find that
  $\varphi^{-1}\in \langle \varphi \rangle\subseteq R$, where $
  \langle \varphi \rangle$ is the span of $\varphi$ with respect to
  $\circ$.
\end{proof}

If $(R,\vee,\circ)$ is a subsemiring of $(\Res(\bL),\vee,\circ)$, then
$(\Mat_{I\times I}(R),+,\cdot)$ is also a subsemiring of
$(\Mat_{I\times I}(\Res(\bL)),+,\cdot)$. The next corollary states the
corresponding result.

\begin{corollary}
  Let $\bL$ be a finite lattice, $(R,\vee,\circ)$ a subsemiring of
  $(\Res(\bL),\vee,\circ)$, $I$~a finite index set, and $A\in
  \Mat_{I\times I}(R)$ such that $A$ is invertible in $\Mat_{I\times
    I}(\Res(\bL))$. Then $A^{-1}\in \Mat_{I\times I}(R)$.
\end{corollary}

This means that for matrices over a subsemiring of $\Res(L)$ one can
also apply Corollary~\ref{corollary_inv_mat} to decide whether a
matrix is invertible and Proposition~\ref{prop_the_inverse} to
construct the inverse of an invertible matrix.  Consequently, one can
do this for every finite additively idempotent semiring with zero and
one by Proposition~\ref{prop_add_idemp_semirings}.  In particular,
these results apply to every proper finite simple semiring with zero
by Theorem~\ref{theorem_zum}.

\subsection{Remarks}

In the following let $(R,+,\cdot)$ be a finite additively idempotent
semiring with zero and one.  To apply
Corollary~\ref{corollary_inv_mat} and
Proposition~\ref{prop_the_inverse} for matrices over~$R$, it is
necessary to represent the semiring as a semiring of residuated
mappings of a finite lattice $\bL$.  Additionally, it is required to
know the representation of the lattice as a direct product
$\bL=\bigtimes_{t\in T} \bL_t $ of irreducible lattices $\bL_t$ and to
represent every residuated mapping (semiring element) as a mapping of
$\bigtimes_{t\in T} \bL_t$.  For example, one can represent
$(R,+,\cdot)$ as the subsemiring $(T(R),\vee,\circ)$ of
$(\Res(\bR),\vee,\circ)$, where \mbox{$\bR=(R,\leq)$} (see
Proposition~\ref{prop_add_idemp_semirings}).  Also in this case, one
has to represent $\bR$ as a direct product $\bR=\bigtimes_{t\in T}
\bR_t$ of irreducible lattices $\bR_t$, and one has to represent every
mapping in $T(R)$ as a mapping of $\bigtimes_{t\in T} \bR_t$.

If the lattice $\bL$ is irreducible, then we know by
Corollary~\ref{cor_gen_perm_mat} that a matrix is invertible iff it is
a generalized permutation matrix.  In this case, determining whether a
matrix is invertible as well as inverting is very easy.  In
particular, if the lattice $\bR$ is irreducible, then a matrix is
invertible iff it is a generalized permutation matrix. Furthermore,
the lattice $\bR$ is irreducible iff the semigroup $(R,+)$ is
irreducible. Hence, we get the following corollary.

\begin{corollary}
  Let $(R,+)$ be irreducible and $A\in \Mat_{I\times I}(R)$. Then $A$
  is invertible iff $A$ is a generalized permutation matrix.
\end{corollary}

If $\bL$ is given without the representation as a direct product of
irreducible lattices, then it may actually be involved to find such a
representation.  In particular, it can be hard to find such a
representation for $\bR$.  In the cryptographic application described
in \cite{MMR} it may be sensible for the involved parties of the
protocol (Alice and Bob) to agree in the setup phase on a random
finite simple semiring by choosing randomly a finite lattice $\bK$ and
taking $(\Res(\bK),\vee,\circ)$ (or a certain subsemiring
$(R,\vee,\circ)$ with $e_{a,b}\in R$ for all $a,b\in K$) as the
semiring. If one picks randomly a finite lattice, then the chosen
lattice is likely to be irreducible, and thus determining whether a
matrix over this semiring is invertible and computing the inverse of
an invertible matrix gets again very easy, since all invertible
matrices are in this case generalized permutation matrices.

In order to prevent that deciding whether a matrix is invertible and
computing the inverse become easy problems, a possible approach is
that the parties agree on several irreducible lattices and publish the
direct product of these lattices, without showing the representation
of this lattice as a direct product.

\end{document}